\documentclass{amsart}

\usepackage{amssymb}
\usepackage{amsmath}
\usepackage{amsthm}
\usepackage{amsfonts}

\usepackage{a4wide}
\usepackage{longtable}
\usepackage{multirow}

\newtheorem{theorem}{Theorem}[section]
\newtheorem{lemma}{Lemma}[section]
\newtheorem{proposition}{Proposition}[section]
\newtheorem{corollary}{Corollary}[section]

\theoremstyle{definition}
\newtheorem{definition}{Definition}[section]
\newtheorem{remark}{Remark}[section]
\newtheorem{example}{Example}[section]
\newtheorem{algorithm}{Algorithm}[section]
\newtheorem{conjecture}{Conjecture}[section]

\begin{document}

\title{\(p\)-capitulation over number fields with \(p\)-class rank two}

\author{Daniel C. Mayer}
\address{Naglergasse 53\\8010 Graz\\Austria}
\email{algebraic.number.theory@algebra.at}
\urladdr{http://www.algebra.at}
\thanks{Research supported by the Austrian Science Fund (FWF): P 26008-N25}

\subjclass[2000]{Primary 11R37, 11R29, 11R11, 11R16; Secondary 20D15}
\keywords{Hilbert \(p\)-class field tower, maximal unramified pro-\(p\) extension, \(p\)-capitulation of class groups,
real quadratic fields, totally real cubic fields;
finite \(3\)-groups, abelianization of type \((3,3)\).}

\date{May 12, 2016}

\begin{abstract}
Theoretical foundations of a new algorithm for determining the \(p\)-capitulation type \(\varkappa(K)\)
of a number field \(K\) with \(p\)-class rank \(\varrho=2\) are presented.
Since \(\varkappa(K)\) alone is insufficient for identifying the second \(p\)-class group
\(\mathfrak{G}=\mathrm{Gal}(\mathrm{F}_p^2{K}\vert K)\) of \(K\),
complementary techniques are developed for finding the nilpotency class and coclass of \(\mathfrak{G}\).
An implementation of the complete algorithm in the computational algebra system Magma
is employed for calculating the Artin pattern \(\mathrm{AP}(K)=(\tau(K),\varkappa(K))\)
of all \(34\,631\) real quadratic fields \(K=\mathbb{Q}(\sqrt{d})\)
with discriminants \(0<d<10^8\) and \(3\)-class group of type \((3,3)\).
The results admit extensive statistics of the second \(3\)-class groups
\(\mathfrak{G}=\mathrm{Gal}(\mathrm{F}_3^2{K}\vert K)\)
and the \(3\)-class field tower groups
\(G=\mathrm{Gal}(\mathrm{F}_3^\infty{K}\vert K)\).
\end{abstract}

\maketitle



\section{Introduction}
\label{s:Intro}

Let \(p\) be a prime number.
Suppose that \(K\) is an algebraic number field
with \(p\)-class group \(\mathrm{Cl}_p{K}:=\mathrm{Syl}_p\mathrm{Cl}_K\)
and \(p\)-\textit{elementary class group} \(\mathrm{E}_p{K}:=\mathrm{Cl}_p{K}\bigotimes_{\mathbb{Z}_p}\mathbb{F}_p\).
By class field theory
\cite[Cor. 3.1, p. 838]{Ma},
there exist precisely \(n:=\frac{p^\varrho-1}{p-1}\) distinct (but not necessarily non-isomorphic)
unramified cyclic extensions \(L_i\vert K\), \(1\le i\le n\), of degree \(p\),
if \(K\) possesses the \(p\)-\textit{class rank} \(\varrho:=\dim_{\mathbb{F}_p}\mathrm{E}_p{K}\).
For each \(1\le i\le n\), let \(j_{L_i\vert K}:\,\mathrm{Cl}_p{K}\to\mathrm{Cl}_p{L_i}\) denote the class extension homomorphism
induced by the ideal extension monomorphism
\cite[\S\ 1, p. 74]{Ma0}.
We let \(U_K\), resp. \(U_{L_i}\), be the group of units of \(K\), resp. \(L_i\).

\begin{proposition}
\label{prp:CapitulationKernel}
(Order of \(\ker j_{L_i\vert K}\))\\
The kernel \(\ker j_{L_i\vert K}\) of the class extension homomorphism
associated with an unramified cyclic extension \(L_i\vert K\) of degree \(\lbrack L_i:K\rbrack=p\)
is a subgroup of the \(p\)-elementary class group \(\mathrm{E}_p{K}\)
and has the \(\mathbb{F}_p\)-dimension

\begin{equation}
\label{eqn:CapitulationOrder}
1\le \dim_{\mathbb{F}_p}\ker j_{L_i\vert K}=\log_p\left(\lbrack L_i:K\rbrack\cdot (U_K:\mathrm{Norm}_{L_i\vert K}U_{L_i})\right)\le\varrho.
\end{equation}

\end{proposition}

\begin{proof}
The proof of the inclusion \(\ker j_{L_i\vert K}\le\mathrm{E}_p{K}\) was given in
\cite[\S\ 1, p. 74]{Ma0}
for \(p=3\), and generally in
\cite[Prop. 4.3.(1), p. 484]{Ma1}.
The relation
\(\#\ker j_{L_i\vert K}=\lbrack L_i:K\rbrack\cdot (U_K:\mathrm{Norm}_{L_i\vert K}U_{L_i})\)
for the unramified extension \(L_i\vert K\) is equivalent to the Theorem on the Herbrand quotient
\cite[Thm. 3, p. 92]{Hb}
and was proved in
\cite[Prop. 4.3, pp. 484--485]{Ma1}.
According to Hilbert's Theorem 94
\cite[p. 279]{Hi},
the kernel \(\ker j_{L_i\vert K}\) cannot be trivial.
\end{proof}

\begin{definition}
\label{dfn:Capitulation}
For each \(1\le i\le n\), the elementary abelian \(p\)-group
\(\ker j_{L_i\vert K}\) is called the \(p\)-\textit{capitulation kernel} of \(L_i\vert K\).
We speak about \textit{total capitulation}
\cite{Cg,CgFt}
if \(\dim_{\mathbb{F}_p}\ker j_{L_i\vert K}=\varrho\),
and \textit{partial capitulation} if \(1\le \dim_{\mathbb{F}_p}\ker j_{L_i\vert K}<\varrho\).
\end{definition}

If \(p\ge 3\) is an odd prime,
and \(K=\mathbb{Q}(\sqrt{d})\) is a quadratic field with fundamental discriminant \(d:=d_K\)
and \(p\)-class rank \(\varrho\ge 1\), then there arise the following possibilities for the \(p\)-capitulation kernel
in any of the unramified cyclic relative extensions \(L_i\vert K\) of degree \(p\),
which are absolutely dihedral extensions \(L_i\vert\mathbb{Q}\) of degree \(2p\),
according to
\cite[Prop. 4.1, p. 482]{Ma1}.

\begin{corollary}
\label{cor:QuadraticField}
(Partial and total \(p\)-capitulation over \(K=\mathbb{Q}(\sqrt{d})\) with \(\varrho\ge 2\))

\begin{equation}
\label{eqn:QuadraticCapitulation}
\dim_{\mathbb{F}_p}\ker j_{L_i\vert K}=
\begin{cases}
1 & \text{ if } K \text{ is complex},\ d<0, \\
1 & \text{ if } K \text{ is real},\ d>0, \text{ and } L_i \text{ is of type } \delta, \\
2 & \text{ if } K \text{ is real},\ d>0, \text{ and } L_i \text{ is of type } \alpha.
\end{cases}
\end{equation}

\noindent
The \(p\)-capitulation over \(K\) is total if and only if \(K\) is real with \(\varrho=2\), and \(L_i\) is of type \(\alpha\).

\end{corollary}

\begin{proof}
In this special case of a quadratic base field \(K\),
the extensions \(L_i\vert K\), \(1\le i\le n\), are pairwise non-isomorphic
although they share a common discriminant which is the \(p\)th power \(d_{L_i}=d_K^p\) of the fundamental discriminant of \(K\)
\cite[Abstract, p. 831]{Ma}.
If \(K\) is complex, the unit norm index equals \(1\),
since the cyclotomic quadratic fields do not possess unramified cyclic extensions of odd prime degree.
If \(K\) is real,
we have \((U_K:\mathrm{Norm}_{L_i\vert K}U_{L_i})=1\) \(\Leftrightarrow\) \(L_i\) is of type \(\delta\),
and \((U_K:\mathrm{Norm}_{L_i\vert K}U_{L_i})=p\) \(\Leftrightarrow\) \(L_i\) is of type \(\alpha\)
\cite[Prop. 4.2, pp. 482--483]{Ma1}.
\end{proof}

The organization of this article is the following.
In \S\
\ref{s:TheorPrereq},
basic theoretical prerequisites for the new capitulation algorithm are developed.
The implementation in Magma
\cite{MAGMA}
consists of a sequence of computational techniques whose actual code is given in \S\
\ref{s:ComputTechn}.
The final \S\
\ref{s:IntprNumerRslt}
demonstrates the results of an impressive application to the case \(p=3\),
presenting statistics of all \(3\)-capitulation types \(\varkappa(K)\),
Artin patterns \(\mathrm{AP}(K)\), and second \(3\)-class groups
\(\mathfrak{G}=\mathrm{Gal}(\mathrm{F}_3^2{K}\vert K)\)
of the \(34\,631\) real quadratic fields \(K=\mathbb{Q}(\sqrt{d})\)
with discriminants \(0<d<10^8\) and \(3\)-class group of type \((3,3)\),
which beats our own records in
\cite[\S\ 6]{Ma1}
and
\cite[\S\ 6]{Ma3}.
Theorems concerning \(3\)-tower groups
\(G=\mathrm{Gal}(\mathrm{F}_3^\infty{K}\vert K)\)
with derived length \(2\le\mathrm{dl}(G)\le 3\)
perfect the current state of the art.



\section{Theoretical prerequisites}
\label{s:TheorPrereq}

In this article, we consider algebraic number fields \(K\)
with \(p\)-class rank \(\varrho=2\), for a given prime number \(p\).
As explained in \S\
\ref{s:Intro},
such a field \(K\) has \(n=p+1\) unramified cyclic extensions \(L_i\) of relative degree \(p\).

\begin{definition}
\label{dfn:ArtinPattern}
By the \textit{Artin pattern} of \(K\) we understand the pair consisting of
the family \(\tau(K)\) of the \(p\)-class groups
of all extensions \(L_1,\ldots,L_n\) as its first component
(called the \textit{transfer target type})
and the \(p\)-\textit{capitulation type} \(\varkappa(K)\) as its second component
(called the \textit{transfer kernel type}),

\begin{equation}
\label{eqn:ArtinPattern}
\mathrm{AP}(K):=\left(\tau(K),\varkappa(K)\right),
\quad \tau(K):=\left(\mathrm{Cl}_p{L_i}\right)_{1\le i\le n},
\quad \varkappa(K):=\left(\ker{j_{L_i\vert K}}\right)_{1\le i\le n}.
\end{equation}

\end{definition}

\begin{remark}
\label{rmk:ArtinPattern}
We usually replace the group objects in the family
\(\tau(G)\), resp. \(\varkappa(G)\),
by ordered abelian type invariants, resp. ordered numerical identifiers
\cite[Rmk. 2.1]{Ma12}.
\end{remark}

We know from Proposition
\ref{prp:CapitulationKernel}
that each kernel \(\ker{j_{L_i\vert K}}\) is a subgroup of the \(p\)-elementary class group \(\mathrm{E}_p{K}\) of \(K\).
On the other hand, there exists a unique subgroup \(S<\mathrm{Cl}_K\) of index \(p\)
such that \(S=\mathrm{Norm}_{L_i\vert K}\mathrm{Cl}_{L_i}\), according to class field theory.
Thus we must first get an overview of the connections between subgroups of index \(p\)
and subgroups of order \(p\) of \(\mathrm{Cl}_K\).


\begin{lemma}
\label{lem:SbgOfIndexP}
Let \(p\) be a prime and \(A\) be a finite abelian group with positive \(p\)-rank
and with Sylow \(p\)-subgroup \(\mathrm{Syl}_p{A}\).
Denote by \(A_0\) the complement of \(\mathrm{Syl}_p{A}\) such that \(A\simeq A_0\times\mathrm{Syl}_p{A}\).
Then any subgroup \(S<A\) of index \(p\) is of the form \(S\simeq A_0\times U\)
with a subgroup \(U<\mathrm{Syl}_p{A}\) of index \(p\).
\end{lemma}

\begin{proof}
Any subgroup \(S\) of \(A\simeq A_0\times\mathrm{Syl}_p{A}\) is of the shape
\(S\simeq S_0\times U\) with \(S_0\le A_0\) and \(U\le\mathrm{Syl}_p{A}\).
We have \(p=(A:S)=(A_0\times\mathrm{Syl}_p{A}:S_0\times U)=(A_0:S_0)\cdot(\mathrm{Syl}_p{A}:U)\).
Since \((A_0:S_0)\) is coprime to \(p\), we conclude that \(S_0=A_0\) and \((\mathrm{Syl}_p{A}:U)=p\).
\end{proof}

An application to the particular case \(A=\mathrm{Cl}_K\) and
\(S=\mathrm{Norm}_{L_i\vert K}\mathrm{Cl}_{L_i}<\mathrm{Cl}_K\)
shows that \(S\simeq(\mathrm{Cl}_K)_0\times U\) with \(U=\mathrm{Norm}_{L_i\vert K}\mathrm{Cl}_p{L_i}<\mathrm{Cl}_p{K}\).


\smallskip
Three cases must be distinguished, according to the abelian type of the \(p\)-class group \(\mathrm{Cl}_p{K}\).
We first consider the general situation of
a finite abelian group \(A\) with type invariants \((a_1,\ldots,a_n)\)
having \(p\)-rank \(r_p(A)=2\), that is,
\(n\ge 2\), \(p\mid a_n\), \(p\mid a_{n-1}\), but \(\gcd(p,a_{i})=1\) for \(i<n-1\).
Then the Sylow \(p\)-subgroup \(\mathrm{Syl}_p{A}\) of \(A\) is of type \((p^u,p^v)\) with integer exponents \(u\ge v\ge 1\),
and the \(p\)-elementary subgroup \(A_p\) of \(A\) is of type \((p,p)\).
We select generators \(x,y\) of \(\mathrm{Syl}_p{A}=\langle x,y\rangle\) such that
\(\mathrm{ord}(x)=p^u\) and \(\mathrm{ord}(y)=p^v\).

\begin{lemma}
\label{lem:OrdPow}
Let \(p\) be a prime number.\\
Suppose that \(G\) is a group
and \(x\in G\) is an element with finite order \(e:=\mathrm{ord}(x)\) divisible by \(p\).\\
Then the power \(x^m\) with exponent \(m:=\frac{e}{p}\) is an element of order \(\mathrm{ord}(x^m)=p\).
\end{lemma}

\begin{proof}
Generally, the order of a power \(x^m\) with exponent \(m\in\mathbb{Z}\)
is given by

\begin{equation}
\label{eqn:OrderOfPower}
\mathrm{ord}(x^m)=\frac{\mathrm{ord}(x)}{\gcd(m,\mathrm{ord}(x))}.
\end{equation}

This can be seen as follows.
Let \(d:=\gcd(m,e)\), and suppose that \(m=d\cdot m_0\) and \(e=d\cdot e_0\), then \(\gcd(m_0,e_0)=1\).
We have \((x^m)^{e_0}=x^{m_0\cdot d\cdot e_0}=(x^e)^{m_0}=1\),
and thus \(n:=\mathrm{ord}(x^m)\) is a divisor of \(e_0\).
On the other hand,
\(1=(x^m)^n=x^{m\cdot n}\), and thus \(e=d\cdot e_0\) divides \(m\cdot n=d\cdot m_0\cdot n\).
Consequently, \(e_0\) divides \(m_0\cdot n\), and thus necessarily \(e_0\) divides \(n\), since \(\gcd(m_0,e_0)=1\).
This yields \(n=e_0\), as claimed.

Finally, put \(m:=\frac{e}{p}\), then \(\mathrm{ord}(x^m)=\frac{e}{\gcd(m,e)}=\frac{e}{\gcd(\frac{e}{p},e)}=\frac{e}{\frac{e}{p}}=p\).
\end{proof}


\noindent
Now we apply Lemma
\ref{lem:OrdPow}
to the situation where
\(A\) is a finite abelian group with type invariants \((a_1,\ldots,a_n)\)
having \(p\)-rank \(r_p(A)=2\), that is,
\(n\ge 2\), \(p\mid a_n\), \(p\mid a_{n-1}\).

\begin{proposition}
\label{prp:pElementarySubgroup}
(\(p\)-elementary subgroup)\\
If \(A\) is generated by \(g_1,\ldots,g_n\),
then the \(p\)-elementary subgroup of \(A\) is given by \(\langle g_{n-1}^{a_{n-1}/p},g_{n}^{a_{n}/p}\rangle\).
\end{proposition}

\begin{proof}
Let generators of \(A\) corresponding to the abelian type invariants \((a_1,\ldots,a_n)\)
be \((g_1,\ldots,g_n)\), in particular, the trailing two generators have orders
\(\mathrm{ord}(g_{n-1})=a_{n-1}\) and \(\mathrm{ord}(g_{n})=a_{n}\) divisible by \(p\).
According to Lemma
\ref{lem:OrdPow},
the powers \(g_{n-1}^{a_{n-1}/p}\) and \(g_{n}^{a_{n}/p}\)
have exact order \(p\) and thus generate the \(p\)-elementary subgroup of \(A\).
\end{proof}


\begin{proposition}
\label{prp:OrderPSubgroups}
(Subgroups of order \(p\))\\
If the \(p\)-elementary subgroup \(A_p=\langle w,z\rangle\) of \(A\) is generated by \(w,z\),
then the subgroups of \(A_p\) of order \(p\) can be given by
\(M_1=\langle z\rangle\) and \(M_i=\langle wz^{i-2}\rangle\) for \(2\le i\le p+1\).
\end{proposition}

\begin{proof}
According to the assumptions, \(A_p\) is elementary abelian of rank \(2\), that is,
of type \((p,p)\), and consists of the \(p^2\) elements \(\lbrace w^iz^j\mid 0\le i,j\le p-1\rbrace\),
in particular, \(w^0z^0=1\) is the neutral element.
A possible selection of generators for the \(\frac{p^2-1}{p-1}=p+1\) cyclic subgroups \(M_i\) of order \(p\)
is to take \(M_1=\langle z\rangle\) and \(M_i=\langle wz^{i-2}\rangle\) for \(2\le i\le p+1\),
since the two cycles of powers of \(wz^{i}\) and \(wz^{j}\) for \(1\le i<j\le p-1\) meet in the neutral element only.
\end{proof}


\begin{proposition}
\label{prp:IndexPandOrderP}
(Connection between subgroups of index \(p\), resp. order \(p\))

\begin{enumerate}

\item
If \(u=v=1\), which is equivalent to \(A_p=\mathrm{Syl}_p{A}\), then\\
\(\lbrace U<\mathrm{Syl}_p{A}\mid (\mathrm{Syl}_p{A}:U)=p\rbrace=\lbrace U<A_p\mid\#U=p\rbrace\).

\item
If \(u>v=1\), then there exists a unique bicyclic subgroup \(\langle x^p,y\rangle\) of index \(p\)
which contains \(A_p\).
The other \(p\) subgroups \(U\) of index \(p\) are cyclic of order \(p^u\),
and they only contain the unique subgroup \(\langle x^{p^{u-1}}\rangle\) of \(A_p\)
generated by the \(p^{u-1}\)th powers.

\item
If \(u\ge v>1\), then each subgroup \(U<\mathrm{Syl}_p{A}\) of index \(p\)
completely contains the \(p\)-elementary subgroup \(A_p\).

\end{enumerate}

\end{proposition}

\begin{proof}
If \(u=v=1\), then \(\mathrm{Syl}_p{A}\simeq (p,p)\simeq A_p\).
Thus, \(p^2=(A_p:1)=(A_p:U)\cdot (U:1)\) implies
\((A_p:U)=(U:1)=p\), for each proper subgroup \(U\).

If \(u>v=1\), then a subgroup \(U\) of index \(p\) is
either of type \((p^u)\), i.e., cyclic,
or of type \((p^{u-1},p)\ge (p,p)\).

If \(u\ge v>1\), then each subgroup \(U\) of index \(p\) is
either of type \((p^u,p^{v-1})>(p,p)\) 
or of type \((p^{u-1},p^v)>(p,p)\).
\end{proof}


\begin{theorem}
\label{thm:TausskyAB}
(Taussky's conditions \(\mathrm{A}\) and \(\mathrm{B}\))\\
Let \(L\vert K\) be an unramified cyclic extension of prime degree \(p\)
of a base field \(K\) with \(p\)-class rank \(\varrho=2\).
Suppose that \(S=\mathrm{Norm}_{L\vert K}\mathrm{Cl}_{L}<\mathrm{Cl}_K\)
and \(U=\mathrm{Norm}_{L\vert K}\mathrm{Cl}_p{L}<\mathrm{Cl}_p{K}\)
are the subgroups of index \(p\) associated with \(L\vert K\),
according to class field theory.\\
Then, we generally have \(\ker j_{L\vert K}\bigcap S=\ker j_{L\vert K}\bigcap U\),
and in particular:

\begin{enumerate}

\item
If \(u=v=1\), then\\
\(L\) is of type \(\mathrm{A}\) if either \(\ker j_{L\vert K}=\mathrm{E}_p{K}\) or \(\ker j_{L\vert K}=U\), and\\
\(L\) is of type \(\mathrm{B}\) if \(\ker j_{L\vert K}\notin\lbrace\mathrm{E}_p{K},U\rbrace\).

\item
If \(u>v=1\), let \(N:=\langle x^p,y\rangle<\mathrm{Cl}_p{K}\) denote the unique bicyclic subgroup of index \(p\), then\\
\(L\) is of type \(\mathrm{A}\) if
either \(\ker j_{L\vert K}=\mathrm{E}_p{K}\) or \(U=N\) or \(U\ne N\) and \(\ker j_{L\vert K}=\langle x^{p^{u-1}}\rangle\), and\\
\(L\) is of type \(\mathrm{B}\) if \(U\ne N\) and \(\ker j_{L\vert K}\notin\lbrace\mathrm{E}_p{K},\langle x^{p^{u-1}}\rangle\rbrace\).

\item
If \(u\ge v>1\), then \(L\) is always of type \(\mathrm{A}\).

\end{enumerate}

\end{theorem}

\begin{proof}
This is an immediate consequence of Proposition
\ref{prp:IndexPandOrderP}.
\end{proof}


\begin{theorem}
\label{thm:Orbits}
(Orbits of TKTs expressing the independence of renumeration)

\begin{enumerate}

\item
If \(u=v=1\), then \(\varkappa\sim\lambda\) if and only if
\(\lambda=\sigma_0^{-1}\circ\varkappa\circ\sigma\) for some permutation \(\sigma\in S_{p+1}\)
and its extension \(\sigma_0\in S_{p+2}\) with \(\sigma_0(0)=0\).

\item
If \(u>v=1\), then \(\varkappa\sim\lambda\) if and only if
\(\lambda=\pi_0^{-1}\circ\varkappa\circ\rho_\ast\) for two permutations \(\pi,\rho\in S_{p}\)
and the extensions \(\pi_0\in S_{p+2}\) with \(\pi_0(0)=0\), \(\pi_0(p+1)=p+1\),
and \(\rho_\ast\in S_{p+1}\) with \(\rho_\ast(p+1)=p+1\).

\item
If \(u\ge v>1\), then \(\varkappa\sim\lambda\) if and only if
\(\lambda=\sigma_0^{-1}\circ\varkappa\circ\tau\) for two permutations \(\sigma,\tau\in S_{p+1}\)
and the extension \(\sigma_0\in S_{p+2}\) with \(\sigma_0(0)=0\).

\end{enumerate}

\end{theorem}

\begin{proof}
The proof for the case \(u=v=1\) was given in 
\cite[p. 79]{Ma0}
and
\cite[Rmk. 5.3, pp. 87--88]{Ma9}.
It is the unique case where subgroups of index \(p\)
coincide with subgroups of order \(p\),
and a renumeration of the former enforces a renumeration of the latter,
expressed by a single permutation \(\sigma\in S_{p+1}\) and its inverse \(\sigma^{-1}\).

If \(u>v=1\), then the distinguished subgroups
\(U_{p+1}=N=\langle x^p,y\rangle\simeq(p^{u-1},p)\) of index \(p\), and
\(V_{p+1}=\langle x^{p^{u-1}}\rangle\) of order \(p\), should have the fixed subscript \(p+1\).
The other \(p\) subgroups \(U_i\), resp. \(V_i\), can be renumerated
completely independently of each other, which can be expressed by two independent
permutations \(\pi,\rho\in S_{p}\).
For details, see
\cite[Rmk. 5.6, p. 89]{Ma9}.

In the case \(u\ge v>1\), finally, the \(p+1\) subgroups of index \(p\) of \(\mathrm{Cl}_p{K}\)
and the \(p+1\) subgroups of order \(p\) of \(\mathrm{Cl}_p{K}\) can be renumerated
completely independently of each other, which can be expressed by two independent
permutations \(\sigma,\tau\in S_{p+1}\).
\end{proof}



\section{Computational techniques}
\label{s:ComputTechn}

In this section, we present the implementation of our new algorithm
for determining the Artin pattern \(\mathrm{AP}(K)\) of a number field \(K\)
with \(p\)-class rank \(\varrho=2\) in MAGMA
\cite{BCP,BCFS,MAGMA},
which requires version V2.21--8 or higher.
Algorithm
\ref{alg:BasisField}
returns the entire class group \(C:=\mathrm{Cl}_K\) of the base field \(K\),
together with an invertible mapping \(mC\) from classes to representative ideals.


\begin{algorithm}
\label{alg:BasisField}
(Construction of the base field \(K\) and its class group \(C\))\\
\textbf{Input:}
The fundamental discriminant \(d\) of a quadratic field \(K=\mathbb{Q}(\sqrt{d})\).\\
\textbf{Code:}
\texttt{
\begin{tabbing}
for \= \kill
SetClassGroupBounds("GRH");\\
if (IsFundamental(d) and not (1 eq d)) then\+\\
   ZX<X> := PolynomialRing(Integers());\\
   P := X\({}\,\hat{}\,{}\)2-d;\\
   K := NumberField(P);\\
   O := MaximalOrder(K);\\
   C,mC := ClassGroup(O);\-\\
end if;
\end{tabbing}
}
\noindent
\textbf{Output:}
The conditional class group \((C,mC)\) of the quadratic field \(K\), assuming the GRH.
\end{algorithm}

\begin{remark}
\label{rmk:BasisField}
By using the statement \texttt{K := QuadraticField(d);}
the quadratic field \(K=\mathbb{Q}(\sqrt{d})\) is constructed directly.
However, the construction by means of a polynomial \(P(X)\in\mathbb{Z}\lbrack X\rbrack\) executes faster
and can easily be generalized to base fields \(K\) of higher degree.
\end{remark}


For the next algorithm it is important to know that
in the MAGMA computational algebra system
\cite{MAGMA},
the composition \(A\times A\to A\), \((x,y)\mapsto x+y\),
of an abelian group \(A\) is written \textit{additively}, and
abelian type invariants \((a_1,\ldots,a_n)\) of a finite abelian group \(A\)
are arranged in \textit{non-decreasing} order \(a_1\le\ldots\le a_n\).


Given the situation in Proposition
\ref{prp:pElementarySubgroup},
where \(A\) is a finite abelian group having \(p\)-rank \(r_p(A)=2\),
Algorithm
\ref{alg:NaturalOrdering}
defines a natural ordering
on the subgroups \(S\) of \(A\) of index \((A:S)=p\) by means of Proposition
\ref{prp:OrderPSubgroups},
if the Sylow \(p\)-subgroup \(\mathrm{Syl}_p{A}\) is of type \((p,p)\).

\begin{algorithm}
\label{alg:NaturalOrdering}
(Natural ordering of subgroups of index \(p\))\\
\textbf{Input:}
A prime number \(p\) and a finite abelian group \(A\) with \(p\)-rank \(r_p(A)=2\).\\
\textbf{Code:}
\texttt{
\begin{tabbing}
for \= for \= for \= \kill
if (2 eq \(\#\)pPrimaryInvariants(A,p)) then\+\\
   n := Ngens(A);\\
   x := (Order(A.(n-1)) div p)*A.(n-1);\\
   y := (Order(A.n) div p)*A.n;\\
   seqS := Subgroups(A: Quot:=[p]);\\
   seqI := [\({}\,{}\)];\\
   for i in [1..p+1] do Append(\(\sim\)seqI,0); end for;\\
   NonCyc := 0;\\
   Cyc := 0;\\
   i := 0;\\
   for S in seqS do\+\\
      i := i+1;\\
      Pool := [\({}\,{}\)];\\
      if y in S\({}\,\grave{}\,{}\)subgroup then Append(\(\sim\)Pool,1); seqI[1] := i; end if;\\
      for e in [0..p-1] do\+\\
         if x+e*y in S\({}\,\grave{}\,{}\)subgroup then Append(\(\sim\)Pool,e+2); seqI[e+2] := i; end if;\-\\
      end for;\\
      if (2 le \(\#\)Pool) then\+\\
         NonCyc := ct;\-\\
      else\+\\
         Cyc := Pool[1];\-\\
      end if;\-\\
   end for;\\
   if (0 lt NonCyc) then\+\\
      for i in [1..p+1] do seqI[i] := i; end for;\-\\
   end if;\-\\
end if;
\end{tabbing}
}
\noindent
\textbf{Output:}
Generators \(x,y\) of the \(p\)-elementary subgroup \(A_p\) of \(A\),
two indicators, \texttt{NonCyc} for one or more non-cyclic maximal subgroups of \(\mathrm{Syl}_p{A}\),
\texttt{Cyc} for one or more cyclic maximal subgroups of \(\mathrm{Syl}_p{A}\),
an ordered sequence \texttt{seqS} of the \(p+1\) subgroups of \(A\) of index \(p\),
and, if there are only cyclic maximal subgroups of \(\mathrm{Syl}_p{A}\),
an ordered sequence \texttt{seqI} of numerical identifiers for the elements \(S\) of \texttt{seqS}.
\end{algorithm}

\begin{proof}
This is precisely the implementation of the Propositions
\ref{prp:pElementarySubgroup},
\ref{prp:OrderPSubgroups}
and
\ref{prp:IndexPandOrderP}
in MAGMA
\cite{MAGMA}.
\end{proof}

\begin{remark}
\label{rmk:HilbertClassField}
The modified statement \texttt{seqS := Subgroups(A: Quot:=[p,p]);}
yields the biggest subgroup of \(A\) of order coprime to \(p\),
and can be used for constructing the Hilbert \(p\)-class field \(\mathrm{F}_p^1{K}\)
of the base field \(K\) in Algorithm
\ref{alg:UnramifiedExtensions},
if the \(p\)-class group \(\mathrm{Cl}_p{K}\) is of type \((p,p)\).
\end{remark}


The class group \((C,mC)\) in the output of Algorithm
\ref{alg:BasisField}
is used as input for Algorithm
\ref{alg:NaturalOrdering}.
The resulting sequence \texttt{seqS} of all subgroups of index \(p\) in \(C\),
together with the pair \((C,mC)\),
forms the input of Algorithm
\ref{alg:UnramifiedExtensions},
which determines all unramified cyclic extensions \(L_i\vert K\) of relative degree \(p\)
using the Artin correspondence as described by Fieker
\cite{Fi}.

\begin{algorithm}
\label{alg:UnramifiedExtensions}
(Construction of all unramified cyclic extensions of degree \(p\))\\
\textbf{Input:}
The class group \((C,mC)\) of a base field \(K\)
and the ordered sequence \texttt{seqS} of all subgroups \(S\) of index \(p\) in \(C\).\\
\textbf{Code:}
\texttt{
\begin{tabbing}
         seqAblExt := [\= \kill
         seqAblExt := [AbelianExtension(Inverse(mQ)*mC)\+\\
                       where Q,mQ := quo<C|S\({}\,\grave{}\,{}\)subgroup>: S in seqS];\-\\
         seqRelFld := [NumberField(ae): ae in seqAblExt];\\
         seqRelOrd := [MaximalOrder(ae): ae in seqAblExt];\\
         seqAbsFld := [AbsoluteField(rf): rf in seqRelFld];\\
         seqAbsOrd := [MaximalOrder(af): af in seqAbsFld];\\
         seqOptRep := [OptimizedRepresentation(af): af in seqAbsFld];\\
         seqOptAbsFld := [NumberField(DefiningPolynomial(opt)): opt in seqOptRep];\\
         seqOptAbsOrd := [Simplify(LLL(MaximalOrder(oaf))): oaf in seqOptAbsFld];
\end{tabbing}
}
\noindent
\textbf{Output:}
Three ordered sequences,
\texttt{seqRelOrd} of the relative maximal orders of \(L_i\vert K\),
\texttt{seqAbsOrd} of the corresponding absolute maximal orders of \(L_i\vert\mathbb{Q}\), and
\texttt{seqOptAbsOrd} of optimized representations for the latter.
\end{algorithm}

\begin{remark}
\label{rmk:UnramifiedExtensions}
Algorithm
\ref{alg:UnramifiedExtensions}
is independent of the \(p\)-class rank \(\varrho\) of the base field \(K\).
In order to obtain the adequate coercion of ideals,
the sequence \texttt{seqRelOrd} must be used for computing the transfer kernel type \(\varkappa(K)\)
in Algorithm
\ref{alg:TKT}.
The trailing three lines of Algorithm
\ref{alg:UnramifiedExtensions}
are optional but highly recommended,
since the size of all arithmetical invariants,
such as polynomial coefficients,
is reduced considerably.
Either the sequence \texttt{seqAbsOrd} or rather the sequence \texttt{seqOptAbsOrd}
should be used for calculating the transfer target type \(\tau(K)\) in Algorithm
\ref{alg:TTT}.
\end{remark}


\begin{algorithm}
\label{alg:TKT}
(Transfer kernel type, \(\varkappa(K)\))\\
\textbf{Input:}
The prime number \(p\), the ordered sequence
\texttt{seqRelOrd} of the relative maximal orders of \(L_i\vert K\),
the class group mapping \(mC\) of the base field \(K\) with \(p\)-class rank \(\varrho=2\),
the generators \(x,y\) of the \(p\)-elementary class group \(E_p\) of \(K\),
and the ordered sequence \texttt{seqI} of numerical identifiers 
for the \(p+1\) subgroups \(S\) of index \(p\) in the class group \(C\) of \(K\).\\
\textbf{Code:}
\texttt{
\begin{tabbing}
         for \= for \= \kill
         TKT := [\({}\,{}\)];\\
         for i in [1..\(\#\)seqRelOrd] do\+\\
            Collector := [\({}\,{}\)];\\
            I := seqRelOrd[i]!!mC(y);\\
            if IsPrincipal(I) then Append(\(\sim\)Collector,seqI[1]); end if;\\
            for e in [0..p-1] do\+\\
               I := seqRelOrd[i]!!mC(x+e*y);\\
               if IsPrincipal(I) then Append(\(\sim\)Collector,seqI[e+2]); end if;\-\\
            end for;\\
            if (2 le \(\#\)Collector) then\+\\
               Append(\(\sim\)TKT,0);\-\\
            else\+\\
               Append(\(\sim\)TKT,Collector[1]);\-\\
            end if;\-\\
         end for;
\end{tabbing}
}
\noindent
\textbf{Output:}
The transfer kernel type TKT of \(K\).
\end{algorithm}

\begin{remark}
\label{rmk:TKT}
In \(2012\), Bembom investigated the \(5\)-capitulation over
complex quadratic fields \(K\) with \(5\)-class group of type \((5,5)\)
\cite[p. 129]{Bb}.
However, his techniques were only able to distinguish between
permutation types and nearly constant types,
since he did not use the crucial sequence of numerical identifiers.
We refined his results in
\cite[\S\ 3.5, 445--451]{Ma4}
by determining the cycle decomposition and,
in particular, the fixed points of the permutation types,
which admitted the solution of an old problem by Taussky
\cite[\S\ 3.5.2, p. 448]{Ma4}.
\end{remark}


\begin{algorithm}
\label{alg:TTT}
(Transfer target type, \(\tau(K)\))\\
\textbf{Input:}
The prime number \(p\) and the ordered sequence
\texttt{seqOptAbsOrd} of the optimized absolute maximal orders of \(L_i\vert\mathbb{Q}\).\\
\textbf{Code:}
\texttt{
\begin{tabbing}
         for \= \kill
         SetClassGroupBounds("GRH");\\
         TTT := [\({}\,{}\)];\\
         for i in [1..\(\#\)seqOptAbsOrd] do\+\\
            CO := ClassGroup(seqOptAbsOrd[i]);\\
            Append(\(\sim\)TTT,pPrimaryInvariants(CO,p));\-\\
         end for;
\end{tabbing}
}
\noindent
\textbf{Output:}
The conditional transfer target type TTT of \(K\), assuming the GRH.
\end{algorithm}

With Algorithms
\ref{alg:TKT}
and
\ref{alg:TTT}
we are in the position to determine the Artin pattern
\(\mathrm{AP}(K)=(\tau(K),\varkappa(K))\) of the field \(K\).
For pointing out fixed points of the transfer kernel type \(\varkappa(K)\)
it is useful to define a corresponding \textit{weak} TKT \(\kappa=\kappa(K)\)
which collects the Taussky conditions A, resp. B, of Theorem
\ref{thm:TausskyAB},
for each extension \(L_i\vert K\):

\begin{equation}
\label{eqn:TAB}
\kappa_i:=
\begin{cases}
\mathrm{A} & \text{ if } \ker{j_{L_i\vert K}}\bigcap\mathrm{Norm}_{L_i\vert K}\mathrm{Cl}_p{L_i}>1, \\
\mathrm{B} & \text{ if } \ker{j_{L_i\vert K}}\bigcap\mathrm{Norm}_{L_i\vert K}\mathrm{Cl}_p{L_i}=1.
\end{cases}
\end{equation}


\begin{algorithm}
\label{alg:TAB}
(Weak transfer kernel type, \(\kappa(K)\), containing Taussky's conditions A, resp. B)\\
\textbf{Input:}
The indicators \texttt{NonCyc}, \texttt{Cyc}, and the TKT.\\
\textbf{Code:}
\texttt{
\begin{tabbing}
         for \= for \= for \= for \= \kill
         TAB := [\({}\,{}\)];\\
         if (0 lt NonCyc) then\+\\
            if (1 eq NonCyc) then\+\\
               for i in [1..\(\#\)TKT] do\+\\
                  if ((Cyc eq TKT[i]) and not (NonCyc eq i))\\
                  or (NonCyc eq i) or (0 eq TKT[i]) then\+\\
                     Append(\(\sim\)TAB,"A");\-\\
                  else\+\\
                     Append(\(\sim\)TAB,"B");\-\\
                  end if;\-\\
               end for;\-\\
            else\+\\
               for i in [1..\(\#\)TKT] do Append(\(\sim\)TAB,"A"); end for;\-\\
            end if;\-\\
         else\+\\
            for i in [1..\(\#\)TKT] do\+\\
               if (i eq TKT[i]) or (0 eq TKT[i]) then\+\\
                  Append(\(\sim\)TAB,"A");\-\\
               else\+\\
                  Append(\(\sim\)TAB,"B");\-\\
               end if;\-\\
            end for;\-\\
         end if;
\end{tabbing}
}
\noindent
\textbf{Output:}
The weak transfer kernel type TAB of \(K\).
\end{algorithm}

\begin{proof}
This is the implementation of Theorem
\ref{thm:TausskyAB}
in MAGMA
\cite{MAGMA}.
\end{proof}



\section{Interpretation of numerical results}
\label{s:IntprNumerRslt}

By means of the algorithms in \S\
\ref{s:ComputTechn},
we have computed the Artin pattern \(\mathrm{AP}(K)=(\tau(K),\varkappa(K))\)
of all \(34\,631\) real quadratic fields \(K=\mathbb{Q}(\sqrt{d})\)
with \(\mathrm{Cl}_3{K}\simeq (3,3)\)
in the range \(0<d<10^8\) of fundamental discriminants.
The results are presented in the following four tables,
arranged by the coclass \(\mathrm{cc}(\mathfrak{G})\) of the second \(3\)-class group \(\mathfrak{G}=\mathrm{G}_3^2{K}\).
Each table gives
the type designation, distinguishing ground states and excited states \((\uparrow,\uparrow^2,\ldots)\),
the transfer kernel type \(\varkappa=\varkappa(K)\),
the transfer target type \(\tau=\tau(K)\),
the absolute frequency AF,
the relative frequency RF, that is the percentage with respect to the total number of occurrences of the fixed coclass,
and the minimal discriminant MD
\cite[Dfn. 5.1]{Ma11}.
Additionally to this experimental information,
we have identified the group \(\mathfrak{G}\) by means of the
strategy of pattern recognition via Artin transfers
\cite[\S\ 4]{Ma12},
and computed the factorized order of its automorphism group \(\mathrm{Aut}(\mathfrak{G})\)
and its relation rank \(d_2(\mathfrak{G}):=\dim_{\mathbb{F}_p}\mathrm{H}^2(\mathfrak{G},\mathbb{F}_p)\).
Groups are specified by their names in the SmallGroups Library
\cite{BEO1,BEO2}.
The nilpotency class \(c=\mathrm{cl}(\mathfrak{G})\) and coclass \(r=\mathrm{cc}(\mathfrak{G})\)
were determined by means of
\cite[Thm. 3.1, p. 290, and Thm. 3.2, p. 291]{Ma7},
resp.
\cite[Thm. 3.1]{Ma11}.



\subsection{Groups \(\mathfrak{G}\) of coclass \(\mathrm{cc}(\mathfrak{G})=1\)}
\label{ss:StatisticsCc1}

The \(31\,088\) fields
whose second \(3\)-class group \(\mathfrak{G}\) is of maximal class, i.e. of coclass \(\mathrm{cc}(\mathfrak{G})=1\),
constitute a contribution of \(89.77\%\), which is dominating by far.
This confirms the tendency which was recogized for the restricted range \(0<d<10^7\) already,
where we had  \(\frac{2\,303}{2\,576}\approx 89.4\%\) in
\cite[Tbl. 2, p. 496]{Ma1}
and
\cite[Tbl. 6.1, p. 451]{Ma3}.
However, there is a slight \textit{increase} of \(0.37\%\) for the relative frequency of \(\mathrm{cc}(\mathfrak{G})=1\)
in the extended range.


\begin{theorem}
\label{thm:Cc1}
(Coclass \(1\))
The Hilbert \(3\)-class field tower
of a real quadratic field \(K\) whose second \(3\)-class group
\(\mathfrak{G}=\mathrm{Gal}(\mathrm{F}_3^2{K}\vert K)\)
is of coclass \(\mathrm{cc}(\mathfrak{G})=1\)
has exact length \(\ell_3{K}=2\), that is,
the \(3\)-class tower group \(G=\mathrm{Gal}(\mathrm{F}_3^\infty{K}\vert K)\)
is isomorphic to \(\mathfrak{G}\), and
\(K<\mathrm{F}_3^1{K}<\mathrm{F}_3^2{K}=\mathrm{F}_3^\infty{K}\).
\end{theorem}

\begin{proof}
This is Theorem 5.3 in
\cite{Ma11}.
\end{proof}

In Table
\ref{tbl:StatisticsCc1},
we denote two crucial mainline vertices of the unique coclass-\(1\) tree \(\mathcal{T}^1(\langle 3^2,2\rangle)\) by
\(M_7:=\langle 3^7,386\rangle\) and \(M_9:=M_7(-\#1;1)^2\),
and we give the results for \(\mathrm{cc}(\mathfrak{G})=1\).



\renewcommand{\arraystretch}{1.1}

\begin{table}[ht]
\caption{Statistics of \(3\)-capitulation types \(\varkappa=\varkappa(K)\) of fields \(K\) with \(\mathrm{cc}(\mathfrak{G})=1\)}
\label{tbl:StatisticsCc1}
\begin{center}
\begin{tabular}{|c|c|c||r|r|r||c|c|r|}
\hline
              Type & \(\varkappa\) &        \(\tau\) &          AF &          RF &               MD &     \(\mathfrak{G}=\mathrm{G}_3^2\) & \(\#\mathrm{Aut}\) & \(d_2\) \\
\hline
           a.\(1\) &      \(0000\) & \(2^2,(1^2)^3\) &  \(2\,180\) &  \(7.01\%\) &      \(62\,501\) & \(\langle 3^6,99\ldots 101\rangle\) &         \(2^13^8\) &   \(3\) \\
           a.\(2\) &      \(1000\) &  \(21,(1^2)^3\) &  \(7\,104\) & \(22.85\%\) &      \(72\,329\) &           \(\langle 3^4,10\rangle\) &         \(2^13^5\) &   \(3\) \\
           a.\(3\) &      \(2000\) &  \(21,(1^2)^3\) & \(10\,514\) & \(33.82\%\) &      \(32\,009\) &            \(\langle 3^4,8\rangle\) &         \(2^23^4\) &   \(3\) \\
      a.\(3^\ast\) &      \(2000\) & \(1^3,(1^2)^3\) & \(10\,244\) & \(32.95\%\) &     \(142\,097\) &            \(\langle 3^4,7\rangle\) &         \(2^23^4\) &   \(3\) \\
\hline
   a.\(1\uparrow\) &      \(0000\) & \(3^2,(1^2)^3\) &      \(58\) &  \(0.19\%\) &  \(2\,905\,160\) &               \(M_7-\#1;5\ldots 7\) &      \(2^13^{12}\) &   \(3\) \\
   a.\(2\uparrow\) &      \(1000\) &  \(32,(1^2)^3\) &     \(242\) &  \(0.78\%\) &     \(790\,085\) &           \(\langle 3^6,96\rangle\) &         \(2^13^9\) &   \(3\) \\
   a.\(3\uparrow\) &      \(2000\) &  \(32,(1^2)^3\) &     \(713\) &  \(2.29\%\) &     \(494\,236\) &   \(\langle 3^6,97\vert 98\rangle\) &         \(2^23^8\) &   \(3\) \\
\hline
 a.\(1\uparrow^2\) &      \(0000\) & \(4^2,(1^2)^3\) &       \(3\) &        \(\) & \(40\,980\,808\) &               \(M_9-\#1;5\ldots 7\) &      \(2^13^{16}\) &   \(3\) \\
 a.\(2\uparrow^2\) &      \(1000\) &  \(43,(1^2)^3\) &       \(9\) &  \(0.03\%\) & \(25\,714\,984\) &                       \(M_7-\#1;2\) &      \(2^13^{13}\) &   \(3\) \\
 a.\(3\uparrow^2\) &      \(2000\) &  \(43,(1^2)^3\) &      \(20\) &  \(0.06\%\) & \(10\,200\,108\) &                \(M_7-\#1;3\vert 4\) &      \(2^23^{12}\) &   \(3\) \\
\hline
 a.\(2\uparrow^3\) &      \(1000\) &  \(54,(1^2)^3\) &       \(1\) &        \(\) & \(37\,304\,664\) &                       \(M_9-\#1;2\) &      \(2^13^{17}\) &   \(3\) \\
\hline
\multicolumn{3}{|c||}{Total of \(\mathrm{cc}(\mathfrak{G})=1\)} & \(31\,088\) & \multicolumn{5}{|c|}{\(89.77\%\) with respect to \(34\,631\)} \\
\hline
\end{tabular}
\end{center}
\end{table}


The large scale separation of the types a.\(2\) and a.\(3\), resp. a.\(2\uparrow\) and a.\(3\uparrow\),
in Table
\ref{tbl:StatisticsCc1}
became possible for the first time by our new algorithm. It refines the results in
\cite[Tbl. 2, p. 496]{Ma1}
and
\cite[Tbl. 6.1, p. 451]{Ma3},
and consequently also the frequency distribution in
\cite[Fig. 3.2, p. 422]{Ma4}.

Inspired by Boston, Bush and Hajir's theory of the
statistical distribution of \(p\)-class tower groups of complex quadratic fields
\cite{BBH},
we expect that, in Table
\ref{tbl:StatisticsCc1}
and in view of Theorem
\ref{thm:Cc1},
the asymptotic limit of the relative frequency RF of
realizations of a particular group \(\mathfrak{G}=\mathrm{G}_3^2{K}\simeq G=\mathrm{G}_3^\infty{K}\)
is proportional to the reciprocal of the order \(\#\mathrm{Aut}(\mathfrak{G})\) of its automorphism group.
In particular, we state the following conjecture about three dominating types,
a.\(3^\ast\), a.\(3\) and a.\(2\).

\begin{conjecture}
\label{cnj:StatisticsCc1}
For a sufficiently extensive range \(0<d<B\) of fundamental discriminants,
both, the absolute and relative frequencies
of realizations of the groups \(\langle 3^4,7\rangle\), \(\langle 3^4,8\rangle\) and \(\langle 3^4,10\rangle\),
resp. \(\langle 3^6,97\rangle\), \(\langle 3^6,98\rangle\) and \(\langle 3^6,96\rangle\),
as \(3\)-class tower groups \(\mathrm{G}_3^\infty{K}=\mathrm{G}_3^2{K}\)
of real quadratic fields \(K=\mathbb{Q}(\sqrt{d})\)
satisfy the proportion \(3:3:2\).
\end{conjecture}

\begin{proof}
(Attempt of an explanation)
A heuristic justification of the conjecture is given for the ground states by the relation for reciprocal orders
\[\#\mathrm{Aut}(\langle 3^4,7\rangle)^{-1}=\#\mathrm{Aut}(\langle 3^4,8\rangle)^{-1}=\frac{1}{2^23^4}
=\frac{3}{2}\cdot\frac{1}{2^13^5}=\frac{3}{2}\cdot\#\mathrm{Aut}(\langle 3^4,10\rangle)^{-1},\]
which is nearly fulfilled by
\(10\,244\approx 10\,514\approx\frac{3}{2}\cdot 7\,104\), resp. \(32.95\%\approx 33.82\%\approx\frac{3}{2}\cdot 22.85\%\),
for the bound \(B=10^8\), and disproves our oversimplified conjectures at the end of
\cite[Rmk. 5.2]{Ma12}.\\
For the first excited states, we have the reciprocal orders
\[\#\mathrm{Aut}(\langle 3^6,97\rangle)^{-1}=\#\mathrm{Aut}(\langle 3^6,98\rangle)^{-1}=\frac{1}{2^23^8}
=\frac{3}{2}\cdot\frac{1}{2^13^9}=\frac{3}{2}\cdot\#\mathrm{Aut}(\langle 3^6,96\rangle)^{-1},\]
but here no arithmetical invariants are known for distinguishing between \(\langle 3^6,97\rangle\) and \(\langle 3^6,98\rangle\),
whence we have \(713\approx 3\cdot 242\), resp. \(2.29\%\approx 3\cdot 0.78\%\), with cumulative factor \(2\cdot\frac{3}{2}=3\).
\end{proof}



\subsection{Groups \(\mathfrak{G}\) of coclass \(\mathrm{cc}(\mathfrak{G})=2\)}
\label{ss:StatisticsCc2}

The \(3\,328\) fields
whose second \(3\)-class group \(\mathfrak{G}\) is of second maximal class, i.e. of coclass \(\mathrm{cc}(\mathfrak{G})=2\),
constitute a moderate contribution of \(9.61\%\).
The corresponding relative frequency for the restricted range \(0<d<10^7\)
is \(\frac{260}{2\,576}\approx 10.1\%\), which can be figured out from
\cite[Tbl. 4--5, pp. 498--499]{Ma1}
or, more easily, from
\cite[Tbl. 6.3, Tbl. 6.5, Tbl. 6.7, pp. 452-453]{Ma3}.
So there is a slight \textit{decrease} of \(0.49\%\) for the relative frequency of \(\mathrm{cc}(\mathfrak{G})=2\)
in the extended range.


\begin{theorem}
\label{thm:SectionD}
(Section \(\mathrm{D}\))
The Hilbert \(3\)-class field tower
of a real quadratic field \(K\) whose second \(3\)-class group
\(\mathfrak{G}=\mathrm{Gal}(\mathrm{F}_3^2{K}\vert K)\)
is isomorphic to either of the two Schur \(\sigma\)-groups
\(\langle 3^5,5\rangle\) or \(\langle 3^5,7\rangle\)
has exact length \(\ell_3{K}=2\), that is,
the \(3\)-class tower group \(G=\mathrm{Gal}(\mathrm{F}_3^\infty{K}\vert K)\)
is isomorphic to \(\mathfrak{G}\), and
\(K<\mathrm{F}_3^1{K}<\mathrm{F}_3^2{K}=\mathrm{F}_3^\infty{K}\).
\end{theorem}

\begin{proof}
This statement has been proved by Scholz and Taussky in
\cite[\S\ 3, p. 39]{SoTa}.
It has been confirmed with different techniques by Brink and Gold in
\cite[Thm. 7, pp. 434--435]{BrGo},
and by Heider and Schmithals in
\cite[Lem. 5, p. 20]{HeSm}.
All three proofs were expressed for complex quadratic base fields \(K\),
but since the cover
\cite[Dfn. 5.1, p. 30]{Ma10}
of a Schur \(\sigma\)-group \(\mathfrak{G}\) consists of a single element,
\(\mathrm{cov}(\mathfrak{G})=\lbrace\mathfrak{G}\rbrace\),
the statement is actually valid for any algebraic number field \(K\),
in particular also for a real quadratic field \(K\).
\end{proof}


\noindent
Table
\ref{tbl:StatisticsCc2}
shows the computational results for \(\mathrm{cc}(\mathfrak{G})=2\),
using the relative identifiers of the ANUPQ package
\cite{GNO}
for groups \(\mathfrak{G}\) of order \(\#\mathfrak{G}\ge 3^8\),
resp. \(G\) of order \(\#G\ge 3^8\).
The possibilities for the \(3\)-class tower group \(G\)
are complete for the TKTs c.\(18\), c.\(21\), E.\(6\), E.\(8\), E.\(9\) and E.\(14\),
constituting the cover of the corresponding metabelian group \(\mathfrak{G}\).
For the TKTs c.\(18\uparrow\), c.\(21\uparrow\),
the cover \(\mathrm{cov}(\mathfrak{G})\) is given in
\cite[Cor. 7.1, p. 38, and Cor. 8.1, p.48]{Ma10},
and for E.\(6\uparrow\), E.\(8\uparrow\), E.\(9\uparrow\) and E.\(14\uparrow\),
it has been determined in
\cite[Cor 21.3, p. 187]{Ma6}.
A selection of densely populated vertices is given for the sporadic TKTs G.\(19^\ast\) and H.\(4^\ast\),
according to
\cite[Tbl. 4--5]{Ma11}.
We denote two important branch vertices of depth \(1\) by
\(N_{9,j}:=\langle 3^7,303\rangle-\#1;1-\#1;j\) for \(j\in\lbrace 3,5\rbrace\).



\renewcommand{\arraystretch}{1.1}

\begin{table}[ht]
\caption{Statistics of \(3\)-capitulation types \(\varkappa=\varkappa(K)\) of fields \(K\) with \(\mathrm{cc}(\mathfrak{G})=2\)}
\label{tbl:StatisticsCc2}
\begin{center}
\begin{tabular}{|c|c|c||r|r|r||c|c|r|}
\hline
              Type & \(\varkappa\) &          \(\tau\) &         AF &          RF &               MD &     \(\mathfrak{G}=\mathrm{G}_3^2\) &     \(\#\mathrm{Aut}\) & \(d_2\) \\
                   &               &                   &            &             &                  &           \(G=\mathrm{G}_3^\infty\) &                        &         \\
\hline
          c.\(18\) &      \(0313\) & \(2^2,21,1^3,21\) &    \(347\) &  \(10.4\%\) &     \(534\,824\) &           \(\langle 3^6,49\rangle\) &             \(2^23^8\) &   \(4\) \\
                   &               &                   &            &             &                  & \(\langle 3^7,284\vert 291\rangle\) & \(2^23^8\vert 2^13^8\) &   \(3\) \\
          c.\(21\) &      \(0231\) &    \(2^2,(21)^3\) &    \(358\) &  \(10.8\%\) &     \(540\,365\) &           \(\langle 3^6,54\rangle\) &             \(2^23^8\) &   \(4\) \\
                   &               &                   &            &             &                  & \(\langle 3^7,307\vert 308\rangle\) & \(2^23^8\vert 2^13^8\) &   \(3\) \\
\hline
  c.\(18\uparrow\) &      \(0313\) & \(3^2,21,1^3,21\) &      \(8\) &   \(0.2\%\) & \(13\,714\,789\) &    \(\langle 3^7,285\rangle-\#1;1\) &          \(2^23^{12}\) &   \(4\) \\
  c.\(21\uparrow\) &      \(0231\) &    \(3^2,(21)^3\) &     \(12\) &   \(0.4\%\) &  \(1\,001\,957\) &    \(\langle 3^7,303\rangle-\#1;1\) &          \(2^23^{12}\) &   \(4\) \\
\hline
           D.\(5\) &      \(4224\) & \(1^3,21,1^3,21\) &    \(546\) &  \(16.4\%\) &     \(631\,769\) &            \(\langle 3^5,7\rangle\) &             \(2^23^6\) &   \(2\) \\
          D.\(10\) &      \(2241\) &  \(21,21,1^3,21\) & \(1\,122\) &  \(33.7\%\) &     \(422\,573\) &            \(\langle 3^5,5\rangle\) &             \(2^13^6\) &   \(2\) \\
\hline
           E.\(6\) &      \(1313\) &  \(32,21,1^3,21\) &     \(40\) &   \(1.2\%\) &  \(5\,264\,069\) &          \(\langle 3^7,288\rangle\) &          \(2^13^{10}\) &   \(3\) \\
                   &               &                   &            &             &                  &     \(\langle 3^6,49\rangle-\#2;4\) &          \(2^13^{10}\) &   \(2\) \\
           E.\(8\) &      \(1231\) &     \(32,(21)^3\) &     \(30\) &   \(0.9\%\) &  \(6\,098\,360\) &          \(\langle 3^7,304\rangle\) &          \(2^13^{10}\) &   \(3\) \\
                   &               &                   &            &             &                  &     \(\langle 3^6,54\rangle-\#2;4\) &          \(2^13^{10}\) &   \(2\) \\
           E.\(9\) &      \(2231\) &     \(32,(21)^3\) &     \(83\) &   \(2.5\%\) &     \(342\,664\) & \(\langle 3^7,302\vert 306\rangle\) &          \(2^13^{10}\) &   \(3\) \\
                   &               &                   &            &             &                  & \(\langle 3^6,54\rangle-\#2;2\vert 6\) &       \(2^13^{10}\) &   \(2\) \\
          E.\(14\) &      \(2313\) &  \(32,21,1^3,21\) &     \(63\) &   \(1.9\%\) &  \(3\,918\,837\) & \(\langle 3^7,289\vert 290\rangle\) &          \(2^13^{10}\) &   \(3\) \\
                   &               &                   &            &             &                  & \(\langle 3^6,49\rangle-\#2;5\vert 6\) &       \(2^13^{10}\) &   \(2\) \\
\hline
   E.\(6\uparrow\) &      \(1313\) &  \(43,21,1^3,21\) &      \(1\) &             & \(75\,393\,861\) & \(\langle 3^7,285\rangle-\#1;1-\#1;4\) &       \(2^13^{14}\) &   \(3\) \\
   E.\(8\uparrow\) &      \(1231\) &     \(43,(21)^3\) &      \(2\) &             & \(26\,889\,637\) & \(\langle 3^7,303\rangle-\#1;1-\#1;2\) &       \(2^13^{14}\) &   \(3\) \\
   E.\(9\uparrow\) &      \(2231\) &     \(43,(21)^3\) &      \(1\) &             & \(79\,043\,324\) & \(\langle 3^7,303\rangle-\#1;1-\#1;4\vert 6\) & \(2^13^{14}\) &  \(3\) \\
  E.\(14\uparrow\) &      \(2313\) &  \(43,21,1^3,21\) &      \(1\) &             & \(70\,539\,596\) & \(\langle 3^7,285\rangle-\#1;1-\#1;5\vert 6\) & \(2^13^{14}\) &  \(3\) \\
\hline
          G.\(16\) &      \(4231\) &     \(32,(21)^3\) &     \(27\) &   \(0.8\%\) &  \(8\,711\,453\) & \(\langle 3^7,301\vert 305\rangle-\#1;4\) &    \(2^23^{12}\) &   \(4\) \\
\hline
  G.\(16\uparrow\) &      \(4231\) &     \(43,(21)^3\) &      \(1\) &             & \(59\,479\,964\) &            \(N_{9,3\vert 5}-\#1;2\) &          \(2^13^{16}\) &   \(4\) \\
\hline
     G.\(19^\ast\) &      \(2143\) &        \((21)^4\) &    \(156\) &   \(4.7\%\) &     \(214\,712\) &           \(\langle 3^6,57\rangle\) &             \(2^43^8\) &   \(4\) \\
                   &               &                   &            &             &                  &          \(\langle 3^7,311\rangle\) &             \(2^23^8\) &   \(3\) \\
\hline
      H.\(4^\ast\) &      \(4443\) &\((1^3)^2,21,1^3\) &    \(493\) &  \(14.8\%\) &     \(957\,013\) &           \(\langle 3^6,45\rangle\) &             \(2^23^8\) &   \(4\) \\
                   &               &                   &            &             &                  & \(\langle 3^7,270\vert 271\rangle\) & \(2^23^8\vert 2^23^9\) &   \(3\) \\
                   &               &                   &            &             &                  & \(\langle 3^7,272\vert 273\rangle\) & \(2^13^9\vert 2^13^8\) &   \(3\) \\
\hline
           H.\(4\) &      \(3313\) &  \(32,21,1^3,21\) &     \(37\) &   \(1.1\%\) &  \(1\,162\,949\) & \(\langle 3^7,286\vert 287\rangle-\#1;2\) &    \(2^23^{12}\) &   \(4\) \\
\hline
\multicolumn{3}{|c||}{Total of \(\mathrm{cc}(\mathfrak{G})=2\)} & \(3\,328\) & \multicolumn{5}{|c|}{\(9.61\%\) with respect to \(34\,631\)} \\
\hline
\end{tabular}
\end{center}
\end{table}


\smallskip
Whereas the sufficient criterion for \(\ell_3{K}=2\) in Theorem
\ref{thm:SectionD}
is known since \(1934\) already,
the following statement of \(2015\) is brand-new and
constitutes one of the few sufficient criteria for \(\ell_3{K}=3\),
that is, for the long desired three-stage class field towers
\cite{BuMa}.

\begin{theorem}
\label{thm:cSection}
(Section \(\mathrm{c}\))
The Hilbert \(3\)-class field tower
of a real quadratic field \(K\) whose second \(3\)-class group
\(\mathfrak{G}=\mathrm{Gal}(\mathrm{F}_3^2{K}\vert K)\)
is one of the six groups \(\langle 3^6,49\rangle\), \(\langle 3^6,54\rangle\),
\(\langle 3^7,285\rangle-\#1;1\), \(\langle 3^7,303\rangle-\#1;1\),
\(\langle 3^7,285\rangle(-\#1;1)^3\), \(\langle 3^7,285\rangle(-\#1;1)^3\)
has exact length \(\ell_3{K}=3\), that is,\\
\(K<\mathrm{F}_3^1{K}<\mathrm{F}_3^2{K}<\mathrm{F}_3^3{K}=\mathrm{F}_3^\infty{K}\).
\end{theorem}

\begin{proof}
This is the union of Thm. 7.1, Cor. 7.1, Cor 7.3, Thm 8.1, Cor 8.1, and Cor 8.3 in
\cite{Ma10}.
\end{proof}


A sufficient criterion for \(\ell_3{K}=3\) similar to Theorem
\ref{thm:cSection}
has been given in
\cite[Thm. 6.1, pp. 751--752]{Ma8}
for \textit{complex} quadratic fields with TKTs in section E.
Due to the relation rank \(d_2\) of the involved groups,
only a weaker statement is possible for \textit{real} quadratic fields with such TKTs.

\begin{theorem}
\label{thm:SectionE}
(Section \(\mathrm{E}\))
The Hilbert \(3\)-class field tower
of a real quadratic field \(K\) whose second \(3\)-class group
\(\mathfrak{G}=\mathrm{Gal}(\mathrm{F}_3^2{K}\vert K)\)
is one of the twelve groups \(\langle 3^7,288\ldots 290\rangle\), \(\langle 3^7,302\vert 304\vert 306\rangle\),
\(\langle 3^7,285\rangle-\#1;1-\#1;4\ldots 6\), \(\langle 3^7,303\rangle-\#1;1-\#1;2\vert 4\vert 6\)
has either length \(\ell_3{K}=3\), that is,
\(K<\mathrm{F}_3^1{K}<\mathrm{F}_3^2{K}<\mathrm{F}_3^3{K}=\mathrm{F}_3^\infty{K}\),
or length \(\ell_3{K}=2\), that is,
\(K<\mathrm{F}_3^1{K}<\mathrm{F}_3^2{K}=\mathrm{F}_3^\infty{K}\).
\end{theorem}

\begin{proof}
This is the union of Thm. 4.1 and Thm. 4.2 in
\cite{Ma11}.
\end{proof}

\begin{example}
That both cases \(\ell_3{K}\in\lbrace 2,3\rbrace\) occur with nearly equal frequency
has been shown for the ground states in Thm. 5.5 and Thm. 5.6 of
\cite{Ma11}.
Due to our extended computations,
we are now in the position to prove that the same is true for the first excited states.
We have \(\ell_3{K}=3\) for the two fields \(K=\mathbb{Q}(\sqrt{d})\) with
\(d=70\,539\,596\), type E.\(14\uparrow\), and \(d=75\,393\,861\), type E.\(6\uparrow\),
but only \(\ell_3{K}=2\) for the three fields with
\(d=79\,043\,324\), type E.\(9\uparrow\), and \(d\in\lbrace 26\,889\,637,\ 98\,755\,469\rbrace\), both of type E.\(8\uparrow\),
\end{example}


Recently, we have provided evidence of asymptotic frequency distributions
for three-stage class field towers,
similar to Conjecture
\ref{cnj:StatisticsCc1}
for two-stage towers.

\begin{conjecture}
\label{cnj:cSection}
For a sufficiently extensive range \(0<d<B\) of fundamental discriminants,
both, the absolute and relative frequencies
of realizations of the groups \(\langle 3^7,284\rangle\) and \(\langle 3^7,291\rangle\),
resp. \(\langle 3^7,307\rangle\) and \(\langle 3^7,308\rangle\)
as \(3\)-class tower groups \(\mathrm{G}_3^\infty{K}=\mathrm{G}_3^3{K}\)
of real quadratic fields \(K=\mathbb{Q}(\sqrt{d})\)
satisfy the proportion \(1:2\).
\end{conjecture}

\begin{proof}
(Attempt of a heuristic justification of the conjecture)\\
For the first two groups, which form the cover of \(\langle 3^6,49\rangle\), we have the reciprocal order relation
\[\#\mathrm{Aut}(\langle 3^7,291\rangle)^{-1}=\frac{1}{2^13^8}
=2\cdot\frac{1}{2^23^8}=2\cdot\#\mathrm{Aut}(\langle 3^7,284\rangle)^{-1},\]
which is nearly fulfilled by the statistical information
\(18\approx 2\cdot 10\), resp. \(64\%\approx 2\cdot 36\%\),
given in
\cite[Thm. 7.2, pp. 34--35]{Ma10}
for \(B=10^7\).\\
For the trailing two groups, which form the cover of \(\langle 3^6,54\rangle\),
only arithmetical invariants of higher order are known for distinguishing between \(\langle 3^7,307\rangle\) and \(\langle 3^7,308\rangle\).
It would have been too time consuming to compute these invariants for
\cite[Thm. 8.2, p. 45]{Ma10}.
\end{proof}


\begin{conjecture}
\label{cnj:SectionH}
For a sufficiently extensive range \(0<d<B\) of fundamental discriminants,
both, the absolute and relative frequencies
of realizations of the groups \(\langle 3^7,270\rangle\), \(\langle 3^7,271\rangle\),
\(\langle 3^7,272\rangle\) and \(\langle 3^7,273\rangle\)
as \(3\)-class tower groups \(\mathrm{G}_3^\infty{K}=\mathrm{G}_3^3{K}\)
of real quadratic fields \(K=\mathbb{Q}(\sqrt{d})\)
satisfy the proportion \(3:1:2:6\).
\end{conjecture}

\begin{proof}
(Attempt of an explanation)
All groups are contained in the cover of \(\langle 3^6,45\rangle\). We have the following relations between reciprocal orders
\[\#\mathrm{Aut}(\langle 3^7,270\rangle)^{-1}=\frac{1}{2^23^8}
=3\cdot\frac{1}{2^23^9}=3\cdot\#\mathrm{Aut}(\langle 3^7,271\rangle)^{-1},\]
\[\#\mathrm{Aut}(\langle 3^7,272\rangle)^{-1}=\frac{1}{2^13^9}
=2\cdot\frac{1}{2^23^9}=2\cdot\#\mathrm{Aut}(\langle 3^7,271\rangle)^{-1},\]
\[\#\mathrm{Aut}(\langle 3^7,273\rangle)^{-1}=\frac{1}{2^13^8}
=3\cdot\frac{1}{2^13^9}=3\cdot\#\mathrm{Aut}(\langle 3^7,272\rangle)^{-1}.\]
Unfortunately,
no arithmetical invariants are known for distinguishing between \(\langle 3^7,271\rangle\) and \(\langle 3^7,272\rangle\).
Therefore, we must replace the two values in the middle of the proportion \(3:1:2:6\)
by a cumulative value \(3:3:6\), resp. \(1:1:2\).
The resulting proportion is fulfilled approximately by the statistical information
\(2\cdot 5\approx 2\cdot 8\approx 11\), resp. \(2\cdot 19\%\approx 2\cdot 29\%\approx 41\%\),
given in
\cite[Thm. 5.7]{Ma11}
for \(B=10^7\).
However, a total of \(24\) individuals cannot be viewed as a statistical ensemble yet.
\end{proof}



\renewcommand{\arraystretch}{1.1}

\begin{table}[hb]
\caption{Statistics of \(3\)-capitulation types \(\varkappa=\varkappa(K)\) of fields \(K\) with \(\mathrm{cc}(\mathfrak{G})=3\)}
\label{tbl:StatisticsCc3}
\begin{center}
\begin{tabular}{|c|c|c||r|r|r||c|c|r|}
\hline
              Type & \(\varkappa\) &            \(\tau\) &      AF &          RF &               MD & \(\mathfrak{G}=\mathrm{G}_3^2\) & \(\#\mathrm{Aut}\) & \(d_2\) \\
\hline
          b.\(10\) &      \(0043\) & \((2^2)^2,(1^3)^2\) &  \(95\) &  \(50.0\%\) &     \(710\,652\) & \(P_7-\#1;21\ldots 26\) & \(2^23^{12}\vert 2^13^{12}\) & \(5\) \\
\hline
  b.\(10\uparrow\) &      \(0043\) & \(3^2,2^2,(1^3)^2\) &   \(6\) &   \(3.2\%\) & \(17\,802\,872\) &         \(P_9-\#1;21\ldots 29\) &      \(2^13^{16}\) &   \(5\) \\
\hline
          d.\(19\) &      \(4043\) &  \(32,2^2,(1^3)^2\) &  \(49\) &  \(26.0\%\) &  \(2\,328\,721\) &            \(P_7-\#1;4\vert 5\) &      \(2^13^{12}\) &   \(5\) \\
          d.\(23\) &      \(1043\) &  \(32,2^2,(1^3)^2\) &  \(16\) &   \(8.4\%\) &  \(1\,535\,117\) &                   \(P_7-\#1;6\) &      \(2^13^{12}\) &   \(5\) \\
          d.\(25\) &      \(2043\) &  \(32,2^2,(1^3)^2\) &  \(22\) &  \(12.0\%\) & \(15\,230\,168\) &            \(P_7-\#1;7\vert 8\) &      \(2^23^{12}\) &   \(5\) \\
\hline
  d.\(19\uparrow\) &      \(4043\) &  \(43,2^2,(1^3)^2\) &   \(1\) &             & \(27\,970\,737\) &            \(P_9-\#1;2\vert 3\) &      \(2^13^{16}\) &   \(5\) \\
  d.\(23\uparrow\) &      \(1043\) &  \(43,2^2,(1^3)^2\) &   \(1\) &             & \(87\,303\,181\) &                   \(P_9-\#1;4\) &      \(2^13^{16}\) &   \(5\) \\
\hline
\multicolumn{3}{|c||}{Total of \(\mathrm{cc}(\mathfrak{G})=3\)} & \(190\) & \multicolumn{5}{|c|}{\(0.55\%\) with respect to \(34\,631\)} \\
\hline
\end{tabular}
\end{center}
\end{table}



\subsection{Groups \(\mathfrak{G}\) of coclass \(\mathrm{cc}(\mathfrak{G})=3\)}
\label{ss:StatisticsCc3}

There are \(190\) fields
whose second \(3\)-class group \(\mathfrak{G}\) is of coclass \(\mathrm{cc}(\mathfrak{G})=3\)
They constitute a very small contribution of \(0.55\%\).
The corresponding relative frequency for the restricted range \(0<d<10^7\)
is \(\frac{10}{2\,576}\approx 0.4\%\), which can be figured out from
\cite[Tbl. 5, p. 499]{Ma1}
or, more easily, from
\cite[Tbl. 6.2, p. 451]{Ma3}.
Thus, there is a slight \textit{increase} of \(0.15\%\) for the relative frequency of \(\mathrm{cc}(\mathfrak{G})=3\)
in the extended range.

For the groups \(\mathfrak{G}\) of coclass \(\mathrm{cc}(\mathfrak{G})\ge 3\),
the problem of determining the corresponding \(3\)-class tower group \(G\)
is considerably harder than for \(\mathrm{cc}(\mathfrak{G})\le 2\), and up to now it is still open.


In Table
\ref{tbl:StatisticsCc3},
we denote two important mainline vertices of the coclass-\(2\) tree \(\mathcal{T}^2(\langle 3^7,64\rangle)\) by
\(P_7:=\langle 3^7,64\rangle\) and \(P_9:=P_7-\#1;3-\#1;1\),
and we give the statistics for \(\mathrm{cc}(\mathfrak{G})=3\).



\subsection{Groups \(\mathfrak{G}\) of coclass \(\mathrm{cc}(\mathfrak{G})=4\)}
\label{ss:StatisticsCc4}

We only have \(25\) fields
whose second \(3\)-class group \(\mathfrak{G}\) is of coclass \(\mathrm{cc}(\mathfrak{G})=4\)
They constitute a negligible contribution of \(0.07\%\).
The corresponding relative frequency for the restricted range \(0<d<10^7\)
is \(\frac{3}{2\,576}\approx 0.1\%\), which can be seen in
\cite[Tbl. 6.9, p. 454]{Ma3}.
So there is a slight \textit{decrease} of \(0.03\%\) for the relative frequency of \(\mathrm{cc}(\mathfrak{G})=4\)
in the extended range.



\renewcommand{\arraystretch}{1.1}

\begin{table}[ht]
\caption{Statistics of \(3\)-capitulation types \(\varkappa=\varkappa(K)\) of fields \(K\) with \(\mathrm{cc}(\mathfrak{G})=4\)}
\label{tbl:StatisticsCc4}
\begin{center}
\begin{tabular}{|c|c|c||r|r|r||c|c|r|}
\hline
            Type & \(\varkappa\) &           \(\tau\) &    AF &       RF &               MD &        \(\mathfrak{G}=\mathrm{G}_3^2\) & \(\#\mathrm{Aut}\) & \(d_2\) \\
\hline
   d.\(25^\ast\) &      \(0143\) & \(3^2,32,(1^3)^2\) & \(4\) & \(16\%\) &  \(8\,491\,713\) &                  \(S_{10,57\vert 59}\) & \(2^23^{16}\) & \(5\) \\
\hline
         F.\(7\) &      \(3443\) & \((32)^2,(1^3)^2\) & \(3\) & \(12\%\) & \(10\,165\,597\) &                         \(P_7-\#2;55\) & \(2^13^{14}\) & \(4\) \\
                 &               &                    &       &          &                  &                 \(P_7-\#2;56\vert 58\) & \(2^23^{14}\) & \(4\) \\
        F.\(11\) &      \(1143\) & \((32)^2,(1^3)^2\) & \(3\) & \(12\%\) & \(66\,615\,244\) &                 \(P_7-\#2;36\vert 38\) & \(2^13^{14}\) & \(4\) \\
        F.\(12\) &      \(1343\) & \((32)^2,(1^3)^2\) & \(6\) & \(24\%\) & \(22\,937\,941\) & \(P_7-\#2;43\vert 46\vert 51\vert 53\) & \(2^13^{14}\) & \(4\) \\
        F.\(13\) &      \(3143\) & \((32)^2,(1^3)^2\) & \(5\) & \(20\%\) &  \(8\,321\,505\) & \(P_7-\#2;41\vert 47\vert 50\vert 52\) & \(2^13^{14}\) & \(4\) \\
\hline
 F.\(7\uparrow\) &      \(3443\) &  \(43,32,(1^3)^2\) & \(1\) &          & \(24\,138\,593\) &                     \(S_{10,39\vert 44}-\#1;5\vert 6\) & \(2^13^{18}\) & \(4\) \\
F.\(12\uparrow\) &      \(1343\) &  \(43,32,(1^3)^2\) & \(1\) &          & \(86\,865\,820\) & \(S_{10,39}-\#1;2\vert 9\), \(S_{10,44}-\#1;3\vert 8\) & \(2^13^{18}\) & \(4\) \\
                 &               &                    &       &          &                  &               \(S_{10,54}-\#1;2\vert 4\vert 6\vert 8\) & \(2^13^{18}\) & \(4\) \\
F.\(13\uparrow\) &      \(3143\) &  \(43,32,(1^3)^2\) & \(1\) &          &  \(8\,127\,208\) & \(S_{10,39}-\#1;3\vert 8\), \(S_{10,44}-\#1;2\vert 9\) & \(2^13^{18}\) & \(4\) \\
                 &               &                    &       &          &                  & \(S_{10,57}-\#1;2\vert 4\), \(S_{10,59}-\#1;3\vert 4\) & \(2^13^{18}\) & \(4\) \\
\hline
        H.\(4\)i &      \(4443\) &  \(43,32,(1^3)^2\) & \(1\) &          & \(54\,313\,357\) &                                          \(T_9-\#1;7\) & \(2^13^{16}\) & \(4\) \\
\hline
\multicolumn{3}{|c||}{Total of \(\mathrm{cc}(\mathfrak{G})=4\)} & \(25\) & \multicolumn{5}{|c|}{\(0.07\%\) with respect to \(34\,631\)} \\
\hline
\end{tabular}
\end{center}
\end{table}


In Table
\ref{tbl:StatisticsCc4},
we denote some crucial mainline vertices of coclass-\(4\) trees \(\mathcal{T}^4(S_{9,j})\) by\\
\(S_{9,j}:=\langle 3^7,64\rangle-\#2;j\) and
\(S_{10,39}:=S_{9,39}-\#1;7\), \(S_{10,44}:=S_{9,44}-\#1;1\), \(S_{10,54}:=S_{9,54}-\#1;8\),\\
\(S_{10,57}:=S_{9,57}-\#1;1\), \(S_{10,59}:=S_{9,59}-\#1;6\),\\
a sporadic vertex by \(T_9:=\langle 3^7,64\rangle-\#2;34\),
and we give the computational results for \(\mathrm{cc}(\mathfrak{G})=4\).

For the essential difference between the location of the groups \(\mathfrak{G}\)
as vertices of coclass trees for the types d.\(25^\ast\) and d.\(25\), see
\cite[Thm. 3.3--3.4 and Exm. 3.1, pp. 490--492]{Ma2}.

The single occurrence of type H.\(4\) belongs to the \textit{irregular variant} (i),
where \(\mathrm{Cl}_3\mathrm{F}_3^1{K}\simeq (9,9,9,9)\).
This is explained in
\cite[p. 498]{Ma1}
and
\cite[pp. 454--455]{Ma3}.
It is the only case in Table
\ref{tbl:StatisticsCc4}
where \(\mathfrak{G}\) is determined uniquely.



\section{Acknowledgements}
\label{s:Acknowledgements}

\noindent
The author gratefully acknowledges that his research is supported
by the Austrian Science Fund (FWF): P 26008-N25.

\newpage


\end{document}